\documentclass{amsart}
\usepackage[utf8]{inputenc}
\usepackage{enumerate}
\usepackage{hyperref}
\usepackage{amsrefs}

\def \rr {\mathbb{R}}
\def \crit{2^{\star}(s)}
\def \crito{2^{\star}}
\def \vv {\mathbb{\Vert}}	
\def \rn {\mathbb{R}^n}
\def \nn {\mathbb{N}}
\def \bb {\hbox}
\def \ep {\epsilon}
\def \ga {\tilde{g}_{\alpha}}
\def \h{H_1^2(M)}
\def \k  {\kappa_{n,s}}
\def \dl{\delta}
\def \kl{K_{\delta}}
\def \la{\lambda}
\def \ld{\lambda_{\delta}}
\def \zd{Z_{\delta}}
\def \ph{\varphi}
\def \px{\Pi_{\kdo^{\perp}}}
\def \non{\nonumber}
\def \cda{\chi_{\al}}
\def \da{\delta_{\alpha}}
\def \pa{\phi_{\alpha}}
\def \vra{\varphi_{\alpha}}
\def \ka{K_{\delta_{\alpha}}^{\perp}}
\def \al{\alpha}
\def \hs{H_1^2}
\def \hsm{H_1^2(M)}
\def \wk{\rightharpoonup}
\def \pta{\widetilde{\phi}_{\alpha}}
\def \eo{\exp_{x_0}}
\def \fp{F^{\,\prime\prime}}
\def \r {\rangle}
\def \l{\langle}
\def \nb{\nabla}
\def \p{\phi_{\delta}}
\def \dg{\Delta_g}
\def \ud{U_{\dl,x_0}}
\def \pr{\partial}	
\def \u{U_{\dl}}
\def \kdo{K_{\dl}}
\def \ldo{L_{\dl,h}}
\def \n{N_{\dl}}
\def \kg{\mathcal{K}_{g}^{n,s}}
\def \fg{\mathcal{F}_g}
\def \sg{\operatorname{Scal}_g}
\def \ks{\Lambda_{n,s}}
\def \eps {\epsilon}
\def \rd {R_\delta}

\def \ricg{\operatorname{Ric}_g}
\def \rmg{\operatorname{Rm}_g}

\def \dundeux{D_1^2(\rn)}

\newtheorem{theo}{Theorem}[section]
\newtheorem{claim}{Claim}[section]
\newtheorem{prop}{Proposition}[section]
\newtheorem{step}{Step}[section]

\title[Ground-state blowing-up solutions for a Hardy-Sobolev equation]{Ground-state blowing-up solutions for a Hardy-Sobolev equation on a manifold}
\author{Hussein Cheikh Ali}
\address{Laboratoire Paul Painlev\'e, Universit\'e de Lille, Cit\'e Scientifique - B\^atiment M2 59655 Villeneuve d'Ascq Cedex France}
\email{frederic.cheikh-ali@univ-lille.fr}
\author{Frédéric Robert}
\address{Institut \'Elie Cartan, Universit\'e de Lorraine, CNRS, IECL, F-54000 Nancy, France}
\email{frederic.robert@univ-lorraine.fr}
\date{July 16th 2025}
\begin{document}
\begin{abstract} We prove the existence of blowing-up families of solutions to an equation of Hardy-Sobolev type in high dimensions. These families are of minimal type. The sole condition is that the potential of the linear operator touches a critical potential at the singular point. This condition is sharp as shown by the first author in \cite{HCA:pjm}.
\end{abstract}
\maketitle
	\section{Introduction}
Let $(M,g)$ be a compact Riemannian manifold of dimension $n\geq 3$ without boundary, $x_0\in M$, $0<s<2$ and $h\in C^0(M)$. We consider solutions $u\in H_1^2(M)\cap C^0(M)$ to the Hardy-Sobolev equation
\begin{equation}\label{equa:HS1}
	 	\left\{\begin{array}{ll}
	 		\Delta_g u +h  u = \frac{u^{\crit -1}}{d_g(x_0,x)^s}& \hbox{  in }M   \\
	 		u > 0  &\hbox{ in }M,
	 	\end{array}\right.
	 \end{equation}
where $\Delta_g:=-\hbox{div}_g(\nabla)$  is the Laplace-Beltrami operator with minus sign convention and $\crit:=\dfrac{2(n-s)}{n-2}$. The space $H_1^2(M)$ be the completion of $C^\infty(M)$ for the norm $\Vert \cdot\Vert_{H_1^2}:=\Vert\nabla \cdot\Vert_2+\Vert \cdot\Vert_2$. The weak formulation of \eqref{equa:HS1} makes sense due to the critical Sobolev embedding  $H^{2}_{1}(M) \hookrightarrow L^{\crit}(M, d_{g}(x_0, \cdot)^{-s})$, see Jaber \cites{JaberNATM,jaber:jmaa}. Our first result is the existence of families of solutions to Hardy-Sobolev equations that blow-up along a standard bubble:
	
\begin{theo}\label{th:2} Let $(M,g)$ be a compact Riemannian manifold of dimension $n\geq 7$. Let $h_0\in C^p(M)$, $p\geq 2$, be such that $\Delta_g+h_0$ is coercive and assume that
	\begin{equation}\label{def:c}
	h_0(x_0)=c_{n,s}\sg(x_0),\hbox{ where }c_{n,s}:=\frac{(n-2)(6-s)}{12(2n-2-s)},
	\end{equation}
	where $\sg$ is the scalar curvature of $(M,g)$. Then there exists $(h_k)_{k\in\nn}\in C^p(M)$ such that $\lim_{k\to +\infty }h_k=h_0$ in $C^p(M)$ and there exists $(u_k)_k\in H_1^2(M)\cap C^0(M)$ such that
	$$\Delta_g u_k+h_k u_k=\frac{u_k^{\crit-1}}{d_g(x,x_0)^s}\, , \, u_k>0\hbox{ in }M$$
and $u_k=Bubble+o(1)$ in $H_1^2(M)$ as $k\to +\infty$, where the Bubble is defined in \eqref{def:Udx_0}.
\end{theo}
The condition on the scalar curvature  is optimal. Indeed, the first author has proved in \cite{HCA:pjm} that if a family blows-up with the profile of a bubble with $n\geq 7$, then $h_0(x_0)=c_{n,s}\sg(x_0)$. The statement is  more precise when a specific geometric quantity does not vanish:
\begin{theo}\label{th:1} Let $(M,g)$ be a compact Riemannian manifold of dimension $n\geq 7$. Let $h_0\in C^p(M)$, $p\geq 2$, be such that $\Delta_g+h_0$ is coercive and 
	$$h_0(x_0)=c_{n,s}\sg(x_0).$$
	Assume that $L_g(h_0,x_0)\neq 0$ where $L_g(h_0,x_0)$ is defined in \eqref{def:Lg:intro}. Let $f\in C^p(M)$ be such that $f(x_0)L_g(h_0,x_0)< 0$.  
	Then there exists $(u_\eps)_{\eps>0}\in H_1^2(M)\cap C^0(M)$ such that
	$$\Delta_g u_\eps+(h_0+\eps f) u_\eps=\frac{u_\eps^{\crit-1}}{d_g(x,x_0)^s}\, , \, u_\eps>0\hbox{ in }M$$
and $u_\eps=Bubble+o(1)$ in $H_1^2(M)$ as $\eps\to 0$, where the Bubble is defined in \eqref{def:Udx_0}.
\end{theo}
The potential $c_{n,s}\sg(x_0)$ first appeared in Jaber \cite{JaberNATM} where it is proved that the condition $h(x_0)<c_{n,s}\sg(x_0)$ permits to obtain minimizing solutions to \eqref{equa:HS1}. Chen \cite{chen} proved that when $h(x_0)>c_{n,s}\sg(x_0)$, there are  families $(u_\eps)_\eps$ of solutions to \eqref{equa:HS1} where the exponent $\crit-1$ is replaced by $\crit-1-\eps$ and where $u_\eps$ blows-up with a bubble profile as $\eps\to 0^+$. Regarding blow-up with several bubbles on a manifold, we refer to CheikhAli-Mazumdar \cite{HCA:SM}.

\smallskip\noindent The quantity $L_g(h_0,x_0)$ depends on the local geometry at the point $x_0$, but also on $s\in (0,2)$, and varies with $s$. Therefore, although it is intrinsic, it does not seem to have any geometric meaning.

\medskip\noindent In the context of a bounded domain of $\rn$, there has been an abundance of studies of Hardy-Sobolev equations. For equations like \eqref{equa:HS1} with the singularity interior to the domain, we refer to Ruiz-Willem \cite{RW} and the recent Ghoussoub-Mazumdar-Robert \cite{gmr:jde} for high energies. It is also possible to add a Hardy potential in the linear operator, which has been studied in Esposito-Ghoussoub-Pistoia-Vaira \cite{egpv} and Ghoussoub-Mazumdar-Robert \cite{gmr:jde}.

\medskip\noindent In order to prove Theorems \ref{th:2} and \ref{th:1}, we use the finite-dimensional reduction. We follow the method developed in Robert-V\'etois \cite{RV:imrn} for the case $s=0$, see also Micheletti-Pistoia-V\'etois \cite{MPV}. For Hardy-Sobolev problems, the linearized problem is one dimensional, which permits to ignore the gradient of the potential. However, the singularity $d_g(\cdot, x_0)^{-s}$ breaks the conformal invariance of the equation on a manifold: this phenomenon does not happen on a domain of $\rn$ since the curvature vanishes. This is why, unlike in the case $s=0$, the classical Yamabe potential $\frac{n-2}{4(n-1)}\sg$ is not relevant in our context. A consequence is that the remainder that comes from the finite-dimensional reduction is not negligible in the energy expansion: we overcome this essential difficulty by performing a delicate analysis of the remainder. A similar phenomenon has been observed in Esposito-Pistoia \cite{EP} and Robert-V\'etois \cite{RV:calcvar}. Note that in \cite{EP} the approach is different: the model for the bubble is a perturbation of ours and the remainder is negligible. In our analysis, we keep the standard bubble and we make a precise estimate of the remainder. 

\medskip\noindent The problem we consider is naturally stated in dimension $n\geq 3$, and, according to CheikhAli \cite{HCA:pjm}, the equality $h_0(x_0)=c_{n,s}\sg(x_0)$ is expected to be the blow-up condition for $n\geq 4$. However, the dimensions $n=4,5,6$ carry some additional complexity. First, for $n=6$, the first term in the expression of $L_g(h_0,x_0)$ can be expressed in term of a log, but the final nonlocal term is implicit and not known, as in Esposito-Pistoia \cite{EP}. For $n=4,5$, one expects to introduce a notion of "mass" as in Robert-V\'etois \cite{RV:imrn}, but, as mentioned above, we cannot exploit the conformal invariance when $s>0$. 

\smallskip\noindent In all this paper, $C(a,b,...)$ will denote any constant depending only on the parameters $a,b,...$. The value might change from one line to the other.

\section{Definition of a bubble and $L_g(h_0,x_0)$}
We fix $0<r_0<i_g(M)/3$ where $i_g(M)>0$ is the injectivity radius of $(M,g)$. Let us choose a cutoff function $\chi \in C^{\infty}(\rr)$ such  that $\chi_{|(-\infty,r_0]}\equiv 1$ and $\chi_{|[2r_0,+\infty)}\equiv 0$. For any $\delta>0$, we define a  {\it bubble} as 
	\begin{eqnarray}\label{def:Udx_0}
 	\ud(x)	&=&\k\, \chi(d_{g}(x_0,x))\frac{\dl^{\frac{n-2}{2}}}{\left( \dl^{2-s}+d_g(x_0,x)^{2-s}\right)^{\frac{n-2}{2-s}} }
		\end{eqnarray}
		for all $x\in M$, where $\k:=\left( (n-s)(n-2)\right) ^{\frac{n-2}{2(2-s)}}$. In particular, we have that
$$\ud(x):=\chi(d_{g}(x_0,x)) U_\delta\left(  \exp_{x_0}^{-1}(x)\right)\hbox{ for all }x\in M$$
where
\begin{equation}\label{def:U0}
U_1(X):=\k\left(\frac{1}{1+|X|^{2-s}} \right)^{\frac{n-2}{2-s}}\hbox{ and }U_\delta(X):=\dl^{-\frac{n-2}{2}}U_1\left( \dl^{-1}X\right)
 \end{equation}
for all $X\in\rn$. It follows from Chou-Chu \cite{Chouchu} that the $U_\delta$'s, $\delta>0$, are the only nontrivial solutions $U\in \dundeux $ (the completion of $C_c^{\infty}(\rr^n)$ for $ \vv\nabla \cdot\vv_2$)   to
 \begin{equation}\label{eq:1}
	\Delta_{\xi}U=\frac{U^{\crit-1}}{|X|^s}, U\geq 0 \bb{ in } \rr^n,
\end{equation}
where $\xi$ is the Euclidean metric. Since $ \Delta_{\xi}U_{\dl}=\frac{U_{\dl}^{\crit-1}}{|X|^s}$ for all $\delta>0$, differentiating with respect to $\delta$ yields
 \begin{equation}\label{eq:Z0}
 \Delta_{\xi}Z_0 =\left(\crit -1 \right) \frac{U_1^{\crit-2}}{|X|^s}Z_0\hbox{ in }\dundeux\hbox{ where }Z_0:=\left( \partial_{\dl}U_{\dl}\right)_{|\delta=1}
 \end{equation}
It follows from Robert \cite{RANA} (see also Dancer-Gladiali-Grossi \cite{DGG}) that 
\begin{equation}\label{spanZ_0}
	K_0:=	\left\lbrace \varphi\in \dundeux: 	\Delta_{\xi}\varphi =\left(\crit -1 \right) \frac{U_1^{\crit-2}}{|X|^s}\varphi \, \bb{ in  } \rr^n\right\rbrace=\operatorname{Span}\{Z_0\}.
\end{equation}
We define
\begin{eqnarray}\label{def:Lg:intro}
L_g(h_0 ,x_0)&:=&-\frac{1}{4n}\left( \dg (h_0-\ks\sg)(x_0)-\kg(x_0)\right) \int_{\rr^n}|X|^4|\nabla U_1|^2\, dx\nonumber\\
&&-\frac{1}{2}  \int_{\rn} W_{s,g,x_0}\hat{C}(W_{s,g,x_0})\, dX 
\end{eqnarray}
with
\begin{equation}\label{def:k}
\kg(x_0):=\frac{\ks}{18}\Big(8|\ricg(x_0)|^2_g-3|\rmg(x_0)|^2_g -\frac{5(2-s)}{10-s}\sg(x_0)^2\Big)
\end{equation}
where $\ks:= \frac{(n-2)(10-s)}{20(2n-2-s)}$ and $\ricg$ and $\rmg$ are the Ricci and Riemann tensors. Note that for $s>0$, $\ks\neq c_{n,s}$, which is a difference with the nonsingular case $s=0$. The function $W_{s,g,x_0}$ is defined as
$$W_{s,g,x_0}:=c_{n,s}\sg (x_0)U_1+\frac{1}{3}R_{ij}(x_0)\sigma^i\sigma^j(r\partial_r U_1)\in L^{\frac{2n}{n+2}}(\rn)\hbox{ for }n\geq 7$$
where $R_{ij}$ are the coordinates of $\ricg$ in the exponential map at $x_0$, and $(r,\sigma)$ are polar coordinates. Finally, $\hat{C}=\hat{C}(W_{s,g,x_0})\in \dundeux$ is the unique solution to 
\begin{equation*}
\Delta_{\xi} \hat{C} -(\crit-1)\frac{U_1^{\crit-2}}{|X|^s}\hat{C} =-W_{s,g,x_0}+\frac{ \int_{\rn}W_{s,g,x_0} Z_0\, dx}{ \int_{\rn}|\nabla Z_0|^2\, dx} \Delta_{\xi}Z_0
\, ;\, \hat{C}\in K_0^\perp,\end{equation*}
where $Z_0$ is as in \eqref{eq:Z0}.

 \section{Preliminary notations} We fix $h_0\in C^2(M)$ such that $\Delta_g+h_0$ is coercive, that is the first eigenvalue $\lambda_1(\Delta_g+h_0)$ is positive. Given $h\in L^\infty(M)$, the space $H_1^2(M)$ is endowed with the bilinear form 
 	\begin{equation}\label{def:norm}
		\langle u, v\rangle_h:= \int_M\Big( (\nabla u,\nabla v)_g+huv\Big)\, dv_g \bb{ for all } u,v\in H_1^2(M).
	\end{equation}
	We let $\ep>0$ be such that for $h\in L^{\infty}(M)$ such that $\vv h-h_0\vv_{\infty}< \ep$. The coercivity of $\dg +h_0$ yields the coercivity of that the operator $\dg +h$  for $\ep<<1$, with a coercivity constant depending only on $ h_0 $ and $\ep<<1$. Therefore, $\langle \cdot,\cdot \rangle_h$ is positive definite and $\left(\hs,  \langle \cdot,\cdot \rangle_h\right) $ is a Hilbert space. 	We let $\vv \cdot \vv_h$ be the norm induced by $\l \cdot, \cdot\r_{h}$: this norm is equivalent to the standard norm $\Vert \cdot\Vert_{H_1^2}$. Solutions to \eqref{equa:HS1} are critical points of the functional 
 \begin{equation*}\label{Jhu}
 	J_h(u):=\frac{1}{2}\int_M\Big( |\nabla u|_g^2+hu^2\Big)\, dv_g-F(u)\bb{ for all } u\in \h,  
 \end{equation*} 
 where $dv_g$ is the Riemannian element of volume and\begin{equation*}
 	F(u):=\frac{1}{\crit} \int_M \frac{u_+^{\crit}}{d_g(x_0,x)^s} \, dv_g \bb{ for all } u\in H_1^2(M),
 \end{equation*} 
with  $u_+:=\max(u,0)$. Indeed, $F,J_h\in C^1(H_1^2(M))$ and 
\begin{equation*}
			J_h^{\prime}(u)[\phi]=\int_M\left( (\nabla u,\nabla \phi)_g +hu\phi\right)\, dv_g- \int_M \frac{u_+^{\crit-1}\phi}{d_g(x_0,x)^s} \, dv_g \bb{ for all } u,\phi\in \h.
		\end{equation*}
We define the  operator $(\Delta_g+h)^{-1}:\left(H_1^2(M)\right)^\prime\to H_1^2(M)$. In particular, for any $w\in L^{\frac{2n}{n+2}}(M)=(L^{\crito}(M))^\prime\subset (H_1^2(M))^\prime$ where $\crito:=\frac{2n}{n-2}$, the function $u=(\Delta_g+h)^{-1}w\in H_1^2(M)$ is the unique solution of $\left( \Delta_g+h\right)u=w$ in $M$. It follows from the continuity of the Sobolev embedding  $H_1^2(M)\hookrightarrow L^{\crito}(M)$, we get for $\eps<<1$ that
	\begin{equation*}
		\vv(\Delta_g+h)^{-1}w\vv_{\hs} \leq C(h_0,\ep)\,  \vv w \vv_{\frac{2n}{n+2}},
	\end{equation*}
where $C(h_0,\ep)>0$  depends only on $(M,g)$, $h_0\in L^{\infty}(M)$, and $0<\ep<<1$. As one checks, differentiating the bubble $\ud$ defined in \eqref{def:Udx_0}, we get
	\begin{equation}\label{prdlud}
		\pr_{\dl} \left( \ud(x)\right)  =\frac{n-2}{2}\k \chi(d_{g}(x_0,x))\dl^{\frac{n-4}{2}}\frac{d_g(x_0,x)^{2-s}-\dl^{2-s}}{\left(\dl^{2-s}+d_g(x_0,x)^{2-s} \right)^{\frac{n-s}{2-s}} }.
	\end{equation}
	For $\dl \in (0,+\infty)$, we define
\begin{equation*}
	\kdo:= \operatorname{Span} \left\lbrace \zd\right\rbrace,
\end{equation*}
where 
	\begin{eqnarray}\label{def:Zdlx_0}
		Z_{\dl}(x)&:=&\chi(d_{g}(x_0,x))\dl^{-\frac{n-2}{2}} Z_0(\dl^{-1}\exp_{x_0}^{-1}(x))\non\\
		&=&\frac{n-2}{2}\k\chi(d_{g}(x_0,x))\dl^{\frac{n-2}{2}}\frac{d_g(x_0,x)^{2-s}-\delta^{2-s}}{\left( \dl^{2-s}+d_g(x_0,x)^{2-s}\right)^{\frac{n-s}{2-s}} }\bb{ for all } x\in M.
	\end{eqnarray}
Let $	\kdo^{\perp}$ as the orthogonal of $\kdo$ in $H_1^2(M)$, defined as 
\begin{equation*}
	\kdo^{\perp}=\{\phi\in H_1^2(M) \, ; \langle \phi , Z_{\delta}\rangle_h=0\}
\end{equation*}
where  $\l \cdot,\cdot\r_h$ is defined in \eqref{def:norm}. We let $\Pi_{\kdo}: H_1^2(M)\mapsto H_1^2(M) $ and $\Pi_{\kdo^{\perp}}: H_1^2(M)\mapsto H_1^2(M)$ be  the orthogonal projection onto respectively $\kdo$ and $\kdo^{\perp}$ with respect to $\l\cdot,\cdot\r_h$.
	\section{The finite dimensional reduction}
The goal of this section is to prove the following:
\begin{theo}\label{maintheo1}
We fix $h_0\in L^{\infty}(M)$ such that $\dg+h_0$ is coercive. Then there exists $\ep >0$ such that for any $h\in L^{\infty}(M)$ satisfying $\vv h-h_0\vv_{\infty}< \ep$, there exists $\delta\mapsto \phi_{\delta}\in C^1\left( (0,\ep), \h\right) $ such that $u_{\dl}:=\ud+\phi_{\dl}$ is a critical point of $J_h$ if and only if $\dl\in (0,\ep)$ is a critical point of $\dl\mapsto J_h(u_{\dl})$ in $(0,\ep)$. Moreover, we have that  $	\vv \phi_{\dl}\vv_{H_1^2} \leq  C\vv \rd\vv_{\frac{2n}{n+2}}$, where $C$ is a constant depending on $(M,g)$, $h_0$ and $\ep$. Here, $\rd$ is defined as: 
		\begin{equation}\label{defreste}
			\rd:= \left( \dg+h\right)\ud -\frac{\ud^{\crit-1} }{d_g(x_0,x)^s}. 
		\end{equation}
		where the peak $\ud$ is given by \eqref{def:Udx_0}.  
	\end{theo}
The proof is classical, and we only stress on the necessary changes due to our context. We prove the Theorem \ref{maintheo1} by constructing  a solutions $u\in H_1^2(M)$ to equation \eqref{equa:HS1}, or equivalently to $u-(\Delta_g+h)^{-1}\left( F^{\,\prime}(u)\right)=0$. This is equivalent to find $u\in H_1^2(M)$ such that
 \begin{eqnarray}
 \Pi_{\kdo} \left(u-(\Delta_g+h)^{-1}\left( F^{\,\prime}(u)\right)\right)=0, \label{system1}\non\\
 \Pi_{\kdo^{\perp}}\left(	u-(\Delta_g+h)^{-1}\left( F^{\,\prime}(u)\right)\right)=0.\label{system2}
 \end{eqnarray}
 We first prove that for $0<\dl<<1$,  small enough, there exists a unique  $\p\in \kdo^{\perp}\cap H_{1}^2(M)$ small enough such that $u_{\delta}=\ud+\phi_{\delta}$ solves  \eqref{system2}. For any $\delta>0$, we introduce the map $\ldo:  \kdo^{\perp} \to \kdo^{\perp} $ defined by 
\begin{equation}\label{ldlxo}
	\ldo(\phi):=\Pi_{K_{\dl}^{\perp}}\left( \phi-(\Delta_g+h)^{-1}(\fp(\ud)\phi)\right) \bb{ for all } \phi \in H_1^2(M).
\end{equation}
We will begin this section by proving the invertibility of $\ldo$.
\begin{prop}\label{propestlphi}
	We fix $h_0\in L^{\infty}(M)$ such that $\dg+h_0$ is coercive. Then, there exists $\ep >0$ such that for any $h\in L^{\infty}(M)$ satisfying $\vv h-h_0\vv_{\infty}< \ep$,
there exists a constant $C>0$ such that for any $x_0\in M$ and any $\dl \in(0,\ep)$, we have that 
	\begin{equation}\label{conditionLphi}
		\vv \ldo(\phi) \vv_{\hs} \geq C \, \vv \phi \vv_{\hs} \bb{ for all }\phi \in H_1^2(M).
	\end{equation}
	In particular, $\ldo$ is a bi-continuous isomorphism.
\end{prop}
\begin{proof} We argue by contradiction. We assume that there exist a sequence $(\da)_\alpha>0$ and a sequence $(h_\alpha)_{\alpha}\in L^\infty(M)$ such that $\da\to 0$, $\Vert h_\alpha-h_0\Vert_\infty\to 0$ as $\alpha \to +\infty$ and a sequence of functions $(\pa)_{\alpha\in \nn}\in \h$ satisfying
\begin{equation}\label{condpa}
 \vv \pa \vv_{H_1^2}=1, \quad	 \pa \in\ka\hbox{ and }\vv L_{\alpha}(\pa)\vv_{H_1^2}\to 0 \bb{ as } \alpha\to +\infty,
\end{equation}
where $L_{\alpha}:=L_{\delta_{\alpha}, h_\alpha}$. For any $\al$, we set $\bar{\chi}_\alpha(X):=\bar{\chi}\left( \dl_{\al} |X|\right)$ where $\bar{\chi}\in C^\infty_c(\rr)$ is such that  $\bar{\chi}_{|(-\infty,2r_0]}\equiv 1$ and
\begin{equation*}
	\pta(X):= \bar{\chi}_\alpha(X) \da^{\frac{n-2}{2}} \pa\circ \eo(\da X) \bb{ for all } X\in  \rr^n.
\end{equation*}
By \eqref{condpa} and with change of variable, we get that the sequence $(\pta)_{\al\in\nn}$ is bounded in $\dundeux$. Passing if necessary to a subsequence, we may assume  that $\pta \wk \widetilde{\phi}$ weakly in $\dundeux$. Set $Z_{\al}:=Z_{\da }$ when $Z_{\da}$ is defined in \eqref{def:Zdlx_0}. Since for any $\al$, the function $\pa$ belongs to $\ka$, and  using a change of variables, we obtain
\begin{eqnarray}\label{prozapa}
0=	\l Z_{\al}, \pa\r_{h}&=&\int_{M} ( \nb Z_{\al}, \nb \pa )_{g}\, dv_g+\int_{M} h_\alpha Z_{\al}\pa\, dv_g\non\\
&=&\int_{\rr^n} (\nb(\cda Z_{0}), \nb \pta )_{g_{\al}}\, dv_{g_{\al}}+\da^2\int_{\rr^n} \tilde{h}_{\al}\cda Z_0\pta \, dv_{g_{\al}},
\end{eqnarray}
where $g_{\al}:=\eo^{\star}g(\da X)$, $\tilde{h}_{\al}(X):=h_\alpha(\eo(\da X))$ for all $X\in \rr^n$,  $Z_0$ as in \eqref{eq:Z0}  and $\cda(X):=\chi(\delta_\alpha |X|)$ for all $X\in\rn$.  For $r_0>0$, by H\"older's inequality and since $\Vert h_\alpha-h_0\Vert_\infty\to 0$, we have that 
\begin{eqnarray*} 
\int_{\rr^n} \tilde{h}_{\al}\cda Z_0\pta \, dv_{g_{\al}}&=&O\left( \int_{B\left( 0,3r_0\delta_{\alpha}^{-1}\right) }|Z_0\pta| \, dv_{g_{\al}}\right)\non\\
 &=&O\left(\left( \int_{B\left( 0,3r_0\delta_{\alpha}^{-1}\right) } |Z_0|^{\frac{2n}{n+2}}\, dX\right)^{\frac{n+2}{n}}\left( \int_{\rr^n } |\widetilde{\phi}_{\alpha}|^{\crito}\, dX\right)^{\frac{2}{\crito}} \right)\non\\ 
 &=&O\left(\vv \pta \vv_{\dundeux}\left( \int_{B\left( 0,3r_0\delta_{\alpha}^{-1}\right) } |Z_0|^{\frac{2n}{n+2}}\, dX\right)^{\frac{n+2}{n}}\right),
\end{eqnarray*}
where the last assertion holds by the continuity of the embedding of $\dundeux$ into $L^{\crito}(\rr^n)$.  Since $\pta$ is bounded in $\dundeux$, it follows from the explicit expression of $Z_0$   that 
\begin{equation*}
	\lim_{\alpha\to +\infty} \delta_{\alpha}^2\int_{\rr^n} \tilde{h}_{\al}\cda Z_0\pta \, dv_{g_{\al}}=0.
\end{equation*} 
This last fact, $\pta \wk \widetilde{\phi}$ weakly in $\dundeux$ and  passing the limit in \eqref{prozapa} yield
\begin{eqnarray}\label{eq:0}
	0=\int_{\rr^n} ( \nb Z_0, \nb \widetilde{\phi})_{\xi} \, dX= (\crit-1)\int_{\rr^n}\frac{U_1^{\crit-2}}{|X|^s}Z_0\widetilde{\phi}\, dX,
\end{eqnarray}
where the last equality holds by \eqref{spanZ_0}. Since $\pa-(\Delta_g+h_\alpha)^{-1}(\fp(U_{\delta_\alpha}))-L_{\al}(\pa)$ belongs to $K_{\delta_{\alpha}}$ that there exists $\la_{\al}\in \rr$ such that 
\begin{equation}\label{initialequation}
	\pa-(\Delta_g+h_\alpha)^{-1}(\fp(U_{\delta_\alpha,x_0})\pa)-L_{\al}(\pa)=\la_{\al} Z_{\al}.
		\end{equation}
	\begin{step}\label{step1}
		We claim that 
		\begin{equation}\label{eq:claim1}
			\vv \pa -(\Delta_g+h_\alpha)^{-1}(\fp(U_{\delta_\alpha, x_0})\pa)\vv_h\to 0 \bb{ as } \al \to +\infty.
		\end{equation}
	\end{step}
	\medskip\noindent{\it Proof of Step \ref{step1}}:
	Indeed, it follows from \eqref{initialequation} that,
	\begin{eqnarray*}
		\vv\pa-(\Delta_g+h_\alpha)^{-1}(\fp(U_{\delta_\alpha, x_0})\pa)\vv_h\leq \vv L_{\al}(\pa)\vv_h+|\la_{\al}|\cdot \vv Z_{\al}\vv_h.
	\end{eqnarray*}
Therefore, in order to establish \eqref{eq:claim1}, it is enough to show, using \eqref{condpa}, that \( \lambda_{\alpha} \to 0 \) as \( \alpha \to +\infty \). For any $\al$, since $\pa, L_{\al}(\pa)\in \ka$ and multiplying \eqref{initialequation} by $Z_{\al}$ we get that
	\begin{eqnarray}\label{eq:01}
		\la_{\al} \vv Z_{\al} \vv_{h}^2&=&-\l (\Delta_g+h_\alpha)^{-1}(\fp(U_{\delta_\alpha, x_0})\pa),Z_{\al}\r_h\non\\
		&=&-(\crit-1)\int_{M}\frac{ (U_{\delta_\alpha, x_0})_+^{\crit-2}}{d_g(x_0,x)^s}Z_{\al}\pa\, dv_g.
	\end{eqnarray} 
By changing the variable, we obtain that
	\begin{eqnarray*}
	 \vv Z_{\al} \vv_{h}^2=\int_{\rr^n} ( \nb(\cda Z_0), \nb(\cda Z_0))_{\xi} \, dv_{g_{\al}}+\da^2\int_{\rr^n} h_{\al}\cda^2 Z_0^2 \, dv_{g_{\al}}.
	\end{eqnarray*} 
	Hence, passing the limit $\al\to +\infty$, and we get that 
	\begin{equation}\label{conzal}
		\vv Z_{\al} \vv_{h}^2\to \vv Z_{0} \vv_{\dundeux}^2>0.
	\end{equation}
Via integration theory, we get that
\begin{eqnarray}\label{restU}
&&\da^{-\frac{n-2}{2}}	\int_{M}\frac{\left( U_{\delta_\alpha}\right)_+^{\crit-2}}{d_g(x_0,x)^s}Z_{\al}\pa\, dv_g=\int_{\rr^n} \frac{\cda U_1^{\crit-2}}{|X|^s}Z_0 \pta\,  dv_{g_{\al}}\non\\
&&\to\int_{\rr^n} \frac{ U_1^{\crit-2}}{|X|^s}Z_0 \tilde{\phi} \,dX=0\hbox{ as }\alpha\to +\infty.
\end{eqnarray}
Then, it remains to inject \eqref{conzal} and \eqref{restU} into \eqref{eq:01}, which gives $\la_{\al}\to 0$ as $\al\to +\infty$. This ends the proof of Step \ref{step1}. \qed \par
\begin{step}\label{step2}
	 We claim that $\lim_{\al\to +\infty} \vv\pa\vv_h=0$.\end{step}
\noindent{\it Proof of Step \ref{step2}:} For any bounded sequence $(\vra)_{\al\in\nn}$ in $\hsm$, we have that 
	\begin{eqnarray}\label{eq:02}
		\l \pa,\vra\r_h&=&\left( \crit-1\right) \int_{M}\frac{\left( U_{\delta_\alpha, x_0}\right)_+^{\crit-2}}{d_g(x_0,x)^s}\pa\vra\, dv_g
	\\
	&&+\l \pa -(\Delta_g+h_\alpha)^{-1}(\fp(U_{\delta_\alpha})\pa),\vra\r_{h_\alpha}.\non
	\end{eqnarray}
By Cauchy-Schwarz inequality and \eqref{eq:claim1}, we get that 
\begin{eqnarray*}
	&&\left|\l \pa -(\Delta_g+h_\alpha)^{-1}(\fp(U_{\delta_\alpha, x_0})\pa),\vra\r\right|\\
	&&\leq \vv \pa -(\Delta_g+h_\alpha)^{-1}(\fp(U_{\delta_\alpha, x_0})\pa)\vv_{\hs} \vv \vra\vv_{\hs}\to 0
\end{eqnarray*}
when $\al\to +\infty$. Thus, returning to \eqref{eq:02}, we obtain that
\begin{equation}\label{scalphivarphi}
	\l \pa,\vra\r_h=\left( \crit-1\right) \int_{M}\frac{\left(  U_{\da,x_0}\right)_+ ^{\crit-2}}{d_g(x_0,x)^s}\pa\vra\, dv_g+ o(1),
\end{equation}
where $o(1)\to 0$ as $\al\to +\infty$. For any smooth function $\varphi$ with compact support in $\rr^n$ and for any $\alpha$, we take 
\begin{equation*}
	\varphi_{\alpha}(x)=\delta_{\alpha}^{\frac{2-n}{2}}\bar{\chi} (x)\varphi\left( \delta_{\al}^{-1}\exp_{x_0}^{-1}(x)\right),
\end{equation*}
where $\bar{\chi}_{|B(x_0,2r_0)}\equiv 1$. Therefore, we get  by a change of variables   that   
\begin{equation}\label{eq:15}
\int_{\rr^n} (\nabla\pta,\nabla\varphi )_{g_\alpha}+\delta_{\al}^2\int_{\rr^n} h_{\alpha} \pta \varphi \, dv_{\ga}=\left( \crit-1\right) \int_{\rr^n}\frac{\left( \cda U_1\right)_+^{\crit-2}}{|X|^s}\pta\varphi\, dv_{g_{\alpha}}+o(1)
\end{equation}
as $\alpha\to +\infty$. By H\"older and Sobolev inequalities, we get that 
\begin{eqnarray*}
	&&\int_{\rr^n \backslash B\left( 0,R\right) }\frac{\left( \cda U_1\right)_+^{\crit-2}}{|X|^s}\pta\varphi dv_{g_{\alpha}}\\
	&=&O\left(\vv \pta\vv_{\dundeux} \vv \varphi\vv_{\dundeux}\left(\int_{\rr^n\backslash B\left( 0,R\right)}\frac{U_1^{\crit}}{|X|^s}\,  dX \right)^{\frac{\crit-2}{\crit}}  \right)\\
	&=& O\left(\eps(R)\vv \pta\vv_{\dundeux} \vv \varphi\vv_{\dundeux}  \right),
\end{eqnarray*}
where $\lim_{R\to \infty}\eps(R)=0$. Letting $\alpha\to +\infty$, and using that $\pta$ is bounded in $\dundeux$, we obtain that 
 \begin{equation}\label{eq:17}
 	\int_{\rr^n}\frac{\left( \cda U_1\right)_+^{\crit-2}}{|X|^s}\pta\varphi\, dv_{g_{\alpha}}=\int_{\rr^n} \frac{  U_1^{\crit-2}}{|X|^s}\pta\varphi\, dv_{g_{\alpha}}+o(1) \bb{ as } \alpha\to +\infty.
 \end{equation}
Plugging this last equation in \eqref{eq:15} yields
\begin{equation*}
	\int_{\rr^n} (\nabla \pta,\nabla\varphi )_{g_\alpha}+\delta_{\al}^2\int_{\rr^n} h_{\alpha} \pta \varphi \, dv_{\ga}=\left( \crit-1\right) \int_{B(0,R)} \frac{   U_1^{\crit-2}}{|X|^s}\pta\varphi\, dv_{g_{\alpha}}+\eps(R)+o(1).
\end{equation*}
Letting $\alpha\to +\infty$, and using and $\pta \wk \widetilde{\phi}$ weakly in $\dundeux$ yields $\widetilde{\phi}\in K_0$ defined in \eqref{eq:Z0}. Hence, using \eqref{eq:0} we get $\widetilde{\phi}\equiv0$. Thus, we have that $\pta \wk 0$ weakly in $\dundeux$. Taking $\vra=\pa$ in \eqref{scalphivarphi}, we get that 
\begin{equation}\label{normphil2}
	\vv \pa\vv_h^2=\left( \crit-1\right) \int_{M}\frac{ U_{\delta_\alpha,x_0}^{\crit-2}}{d_g(x_0,x)^s}\pa^2\, dv_g+o(1) \bb{ as } \al\to +\infty.
\end{equation}
From \eqref{eq:17} and \eqref{condpa}, we obtain that 
\begin{equation*}
	\int_{M}\frac{ U_{\da,x_0}^{\crit-2}}{d_g(x_0,x)^s}\pa^2\, dv_g=\int_{\rr^n} \frac{  U_1^{\crit-2}}{|X|^s}\pta^2\, dv_{g_{\alpha}} +o(1) \bb{ as } \al \to +\infty.
\end{equation*}
Since $U_1\in L^{\crit}(\rr^n,|X|^{-s})$ and $\pta \wk 0$, integration theory yields
\begin{equation*}
	\int_{M}\frac{U_{\da,x_0}^{\crit-2}}{d_g(x_0,x)^s}\pa^2\, dv_g\to 0 \bb{ as } \alpha\to +\infty.
\end{equation*}
This implies with \eqref{normphil2} that $\vv \pa \vv_{h}\to 0$ as $\al\to +\infty$. This ends the proof of Step \ref{step2}. \qed\par
\noindent Finally, we get a contradiction with \eqref{condpa}. This prove 
\eqref{conditionLphi}. It follows from abstract functional analysis that $L_{\delta,h }$  is a bi-continuous isomorphism. This concludes the proof of Proposition \ref{propestlphi}.\end{proof}

\begin{prop}\label{mainprop1}
		We fix $h_0\in L^{\infty}(M)$ such that $\dg+h_0$ is coercive. Then, there exists $\ep >0$ such that for any $h\in L^{\infty}(M)$ satisfying $\vv h-h_0\vv_{\infty}<\ep$ and there exists a constant $C>0$ such that for any $x_0\in M$ and any $\dl \in(0,\ep)$, there exists a unique  $\phi_{\dl} \in C^1((0,\ep),H_1^2(M))$ that solves \eqref{system2}. In addition, we have $\phi_{\dl}\in K_{\dl}^{\perp} $ and 
	\begin{eqnarray*}
\vv \phi_{\dl} \vv_{H_1^2}\leq C\,\left\|  \ud-(\Delta_g+h)^{-1}\left( F^{\,\prime}(\ud)\right)\right\|_{H_1^2}\leq C\vv \rd\vv_{\frac{2n}{n+2}},
	\end{eqnarray*}
	where $C$ is a constant depending on $(M,g)$, $h_0$ and $\ep$. The remainder $R_{\dl}$ is defined in \eqref{defreste}. Moreover, we have that $\vv R_{\dl}\vv_{\frac{2n}{n+2}}\leq c(\dl)$ for all $\dl\in(0,\ep)$, where $\lim_{\dl\to 0} c(\dl)=0$.
\end{prop}
\begin{proof}
 Indeed, for any $\phi\in \kdo^{\perp}\subset H_1^2(M)$, we have that 
\begin{equation}\label{px1}
\px\left( \ud+\phi -(\Delta_g+h)^{-1}\left(F^{\,\prime}\left( \ud+\phi\right)  \right) \right) =0	 
\end{equation}
if and only if $\phi=T(\phi)$, where $T: \kdo^{\perp}\mapsto \kdo^{\perp}$ is such that 
\begin{equation}\label{defTphi}
	T(\phi):=\ldo^{-1}\left(\n(\phi)+\Pi_{K_{\dl}^{\perp}}\left(  (\Delta_g+h)^{-1}\left(F^{\,\prime}\left( \ud\right)\right) -\ud\right)\right),
\end{equation}
where $\ldo$ as defined in \eqref{ldlxo} and 
\begin{eqnarray}
&&	\n(\phi):=\Pi_{K_{\dl}^{\perp}}\left((\Delta_g+h)^{-1}\left(  F^{\,\prime}\left(  \ud+\phi\right)  -F^{\,\prime}\left( \ud\right) -\fp(\ud)\phi\right) \right).\label{nphi}
\end{eqnarray} 
The existence and the $C^1-$regularity of a solution to \eqref{px1} is via Picard's fixed point theorem. This is standard and we refer to Robert-V\'etois \cite{RV:Bangalore} of Micheletti-Pistoia-V\'etois \cite{MPV} for instance. An essential point here is that for $0<\theta<\min\{1,\crit-2\}$, we have that $F\in C^{2,\theta}_{loc}(H_1^2(M))$ and for all $\rho>0$, we have that $\left\|  F\right\|_{C^{2,\theta}(B(0,\rho))}\leq C(\rho,\theta)$ where $C(\rho,\theta)$ is a  positive constant. The limit of $\Vert\rd\Vert_{2n/(n+2)}$ as $\delta\to 0$ is an easy computation. This ends Proposition \ref{mainprop1}.
\end{proof}\par 
\subsection{Proof of Theorem \ref{maintheo1}:} For $\ep>0$ satisfying the hypothesis of Proposition \ref{mainprop1},  there exists $\phi_{\dl}\in C^1\left(  (0,\ep),\h\right) $ such that 
\begin{equation*} 
	\Pi_{\kdo^{\perp}}\left(u_{\dl} -(\Delta_g+h)^{-1}\left(  F^{\,\prime}\left( u_{\dl}\right)\right)\right) =0 \bb{ where }	u_{\dl}:=\ud+\phi_{\dl},
\end{equation*}
and
\begin{eqnarray*} 
\phi_{\dl}\in \kdo^{\perp} \bb{ and }	\vv \phi_{\dl}\vv_{H_1^2}\leq C	\vv R_{\dl}\vv_{\frac{2n}{n+2}}
\end{eqnarray*}
for $\dl\in (0,\ep)$. It then from the definition of $\kl$ there exists $\ld\in\rr$ such that 
\begin{equation*}
\Pi_{\kdo}\left(u_{\dl} -(\Delta_g+h)^{-1}\left(  F^{\,\prime}\left( u_{\dl}\right)\right)\right) =u_{\dl} -(\Delta_g+h)^{-1}\left( F^{\, \prime}\left( u_{\dl}\right) \right) =\ld\zd.	\end{equation*}
This  yields 
\begin{eqnarray}\label{djhu}
	DJ_{h}(u_{\dl})(\ph)=(	u_{\dl} -(\Delta_g+h)^{-1}\left( F^{\, \prime}\left( u_{\delta}\right) \right),\ph)_h=\ld\l\zd,\ph\r_h 
\end{eqnarray}
  for any $\varphi\in \h$. If $u_{\delta}$ is a critical point of $J_{h}$, then $\delta \in B(0,\ep)$ is a critical point for $\delta \mapsto J_h(u_{\delta})$. Conversely, we assume now that $\dl\in (0,\ep)$ is a critical point for the map $\dl\mapsto J_h(u_{\dl})$. We then get that 
\begin{eqnarray}\label{djhu1}
	0=\frac{\partial}{\partial\delta}J_h(u_{\dl})=DJ_{h}(u_{\dl}).\left(\partial_{\dl}\ud+\partial_{\dl}\phi_{\dl} \right).
\end{eqnarray}
It follows from \eqref{djhu} and \eqref{djhu1} that 
\begin{eqnarray}\label{prodzw}
	\ld\left( \l\zd,\partial_{\dl} \ud\r_h+\l\zd,\partial_{\dl}\phi_{\dl}\r_h\right) =0
\end{eqnarray}
Since $\phi_{\dl}\in \kdo^{\perp}$, we have that $\l Z_{\dl},\phi_{\delta}\r_h=0$. Differentiating  with respect to $\dl$, we get that $\l \partial_{\dl}Z_{\dl},\phi_{\dl}\r_h+\l Z_{\dl},\partial_{\dl}\phi_{\dl}\r_h=0$, and therefore \eqref{prodzw} rewrites 
\begin{equation*} 
	\big| \ld\l\zd,\partial_{\dl} U_{\dl, x_0}\r_h \big|  \leq \left| \ld\right| \vv \partial_{\dl}Z_{\dl}\vv_{\hs}\vv \phi_{\dl}\vv_{\hs}.
\end{equation*}
It follows from \eqref{def:Udx_0}, \eqref{prdlud} and \eqref{def:Zdlx_0} that $\partial_{\dl} \ud=\delta^{-1}Z_{\dl}$. Rough computations yield $\delta\vv \partial_\delta Z_{\delta}\vv_{\h}\leq C$. We then get  $|\ld |\cdot \Vert Z_\delta\Vert_{h}^2 \leq |\ld|\cdot C \vv \phi_{\dl}\vv_{\hs}$. Since $\Vert Z_\delta\Vert_{h}^2$ is bounded from below and $\vv \phi_{\dl}\vv_{\hs}\to 0$ as $\delta\to 0$,   for  $0<\delta<<1$, we get that $\ld =0$. This yields $DJ_{h}(u_{\dl})(\ph)=0$ for all $\varphi\in \h$, and then $u_{\dl}$ is a critical point of $J_{h}$ for $\delta \in (0, \ep)$. This ends the proof of Theorem \ref{maintheo1}. \qed

\section{Estimates of the rest}
We fix $h_0\in C^2(M)$ such that $\Delta_g+h_0$ is coercive, that is the first eigenvalue $\lambda_1(\Delta_g+h_0)$ is positive. We fix $C_0>0$ and we consider $h\in C^2(M)$ such that
\begin{equation}\label{hyp:h}
\Vert h\Vert_{C^2}\leq C_0\hbox{ and }\lambda_1(\Delta_g+h)\geq C_0^{-1}.
\end{equation}
In all this section, the controls and the convergences will all be uniform with respect to $h\in C^2(M)$ such that \eqref{hyp:h} holds. Note that for $C_0>0$ large enough depending on $h_0$, there existe $\eps_0>0$ depending also on $h_0$ such that \eqref{hyp:h} is satisfied for all $h\in C^2(M)$ such that $\Vert h-h_0\Vert_{C^2}<\eps_0$. We let $\u$ and $\ud$ be as in \eqref{def:U0} and \eqref{def:Udx_0}. Recall that the remainder is
$$\rd:=(\Delta_g+h)\ud-\frac{\ud^{\crit-1} }{d_g(x,x_0)^s}.$$

\subsection{Approximation of $\rd$} The explicit expression \eqref{def:Udx_0} yields 
$$|\rd(x)|\leq C\delta^{\frac{n-2}{2}}\hbox{ for }x\in M,\, d_g(x,x_0)>i_g(M)/4.$$
Since $\u$ is radially symmetric, we have that
$$\Delta_g\ud=-\frac{\partial_r(r^{n-1}\sqrt{|g|}\partial_r\u)}{r^{n-1}\sqrt{|g|}}=\Delta_\xi \ud-\frac{\partial_r\sqrt{|g|}}{\sqrt{|g|}}\partial_r\ud.$$
Cartan's expansion of the metric yields
$$\frac{\partial_r\sqrt{|g|}}{\sqrt{|g|}}=-\frac{1}{3}R_{ij} r\sigma^i\sigma^j+O(r^2)$$
in Riemannian polar coordinates, where $R_{ij}=R_{ij}(x_0)$ are the coordinates of the Ricci tensor at $x_0$ in the exponential chart at $x_0$. Since \eqref{hyp:h} yields $|h(x)-h(x_0)|\le C_0d_g(x,x_0)$ for all $x\in M$, using that $\u$ is a solution to \eqref{eq:1}, for $X\in \rn$ such that $|X|<i_g(M)/4$, we get
\begin{equation*}
\rd(\hbox{exp}_{x_0}(X))= h(x_0)\u(X)+\frac{1}{3}R_{ij} \sigma^i\sigma^j\delta^{-\frac{n-2}{2}}(r\partial_r U_1)(\delta^{-1}X)+O\left( \frac{\delta^{\frac{n-2}{2}}}{(\delta+|X|)^{n-3}}\right).
\end{equation*}
Finally, we get that
\begin{equation}\label{exp:Rd}
\rd(x)= h(x_0)\ud+\frac{1}{3}R_{ij} \sigma^i\sigma^j\delta^{-\frac{n-2}{2}}(r\partial_r U_1)(\delta^{-1}X)+O\left( \frac{\delta^{\frac{n-2}{2}}}{(\delta+d_g(x,x_0))^{n-3}}\right)
\end{equation}
for all $x\in M$. With the explicit expression of $U_1$ and $\ud$, we then get that
\begin{equation}\label{eq:R:7}
\Vert \rd\Vert_{\frac{2n}{n+2}} =\delta^2 \left\Vert   h(x_0)U_1+\frac{1}{3}R_{ij} \sigma^i\sigma^j (r\partial_r U_1) \right\Vert_{L^{\frac{2n}{n+2}}(\rn)}+o(\delta^2)\hbox{ if }n\geq 7
\end{equation}
We now let $\phi_\delta\in K_\delta^\perp$ be given by Proposition \ref{mainprop1}. Since $\phi_\delta$ is a fixed point for $T$ given by \eqref{defTphi} and bounding roughly  $N_\delta$ defined in \eqref{nphi}, we get that
\begin{eqnarray*}
\phi_\delta&=&	T(\phi_\delta)=\ldo^{-1}\left(\n(\phi)+\Pi_{K_{\dl}^{\perp}}\left(  (\Delta_g+h)^{-1}(-R_\delta)\right)\right)\\
&=&-\ldo^{-1}\left(\Pi_{K_{\dl}^{\perp}}   (\Delta_g+h)^{-1}\left(R_\delta\right)\right)+O(\Vert \rd\Vert_{\frac{2n}{n+2}}^{1+\theta})\\
&=& B_\delta+O(\Vert \rd\Vert_{\frac{2n}{n+2}}^{1+\theta})
\end{eqnarray*}
where $0<\theta<\min\{1, \crit-2\}$ and
\begin{equation}\label{def:B}
B_\delta:=-\ldo^{-1}\left(\Pi_{K_{\dl}^{\perp}}   (\Delta_g+h)^{-1}\left(R_\delta\right)\right)\in K_\delta^\perp
\end{equation}
and $L_{\delta,h}$ is defined in \eqref{ldlxo}. In particular, 
\begin{equation}\label{bnd:B:R}
\Vert B_\delta\Vert_{H_1^2}=O\left(\Vert R_\delta\Vert_{\frac{2n}{n+2}}\right).
\end{equation}
Note that \eqref{px1} rewrites as the existence of $\lambda_\delta\in\rr$ such that $J_h'(\ud+\phi_\delta)=\lambda_\delta\langle Z_\delta,\cdot\rangle_h$. Therefore, a Taylor expansion yields
$$J_h'(\ud+\phi_\delta)=J_h'(\ud)+J_h''(\ud)(\phi_\delta)+O(\Vert\phi_\delta\Vert^{1+\theta}).$$
Since $\phi_\delta\perp Z_\delta$, we then get that
$$J_h''(\ud)(\phi_\delta,\phi_\delta)=-J_h'(\ud)(\phi_\delta)+O(\Vert\phi_\delta\Vert^{2+\theta})$$
Therefore, the Taylor expansion of $J_h$ yields
\begin{eqnarray}
J_h(\ud+\phi_\delta)&=& J_h(\ud)+J_h'(\ud)(\phi_\delta)+\frac{1}{2}J_h''(\ud)(\phi_\delta,\phi_\delta)+ O(\Vert\phi_\delta\Vert^{2+\theta})\non\\
&=& J_h(\ud)-\frac{1}{2}J_h'(\ud)(\phi_\delta)+ O(\Vert\phi_\delta\Vert^{2+\theta})\non\\
&=& J_h(\ud)-\frac{1}{2}\int_M R_\delta\phi_\delta\, dv_g +O(\Vert\phi_\delta\Vert^{2+\theta})\non\\
&=& J_h(\ud)-\frac{1}{2}\int_M R_\delta B_\delta\, dv_g +O(\Vert R_\delta\Vert^{2+\theta})_{\frac{2n}{n+2}}\label{exp:energy}
\end{eqnarray}
Let us recall that $n\geq 7$. Let us assume that  $h(x_0)\neq 0\hbox{ or }\ricg(x_0)\neq 0$. As one checks, using  \eqref{eq:R:7}, as long as $n\geq 7$, we have that
\begin{equation}\label{def:alpha}
 \lim_{\delta\to 0}\frac{\Vert \rd\Vert_{\frac{2n}{n+2}}}{\delta^2}=\alpha^{-1}:=\left\Vert   h(x_0)U_1+\frac{1}{3}R_{ij} \sigma^i\sigma^j (r\partial_r U_1) \right\Vert_{L^{\frac{2n}{n+2}}(\rn)}>0.
 \end{equation}
It follows from the definition \eqref{def:B} that there exists $\mu_\delta\in\rr$ such that\
\begin{equation}\label{eq:B}
\Delta_g B_\delta+hB_\delta-(\crit-1)\frac{\ud^{\crit-2}}{d_g(\cdot,x_0)^s}B_\delta=-\rd+\mu_\delta(\Delta_gZ_\delta+hZ_\delta)
\end{equation}
Multiplying by $Z_\delta$, integrating, using that $B_\delta\perp Z_\delta$ and \eqref{bnd:B:R}, we get that
\begin{eqnarray*}
\mu_\delta\Vert Z_\delta\Vert_{h}^2&=&-(\crit-1)\int_M\frac{\ud^{\crit-2}}{d_g(\cdot,x_0)^s}B_\delta Z_\delta\, dv_g+\int_M R_\delta Z_\delta\, dv_g\\
&=&O\left(\Vert \ud\Vert_{H_1^2}^{\crit-2}\Vert B_\delta\Vert_{H_1^2}\Vert Z_\delta\Vert_{H_1^2}+\Vert \rd\Vert_{\frac{2n}{n+2}}\right)\\
&=& O\left( \Vert B_\delta\Vert_{H_1^2} +\Vert \rd\Vert_{\frac{2n}{n+2}}\right)=O\left( \Vert \rd\Vert_{\frac{2n}{n+2}} \right)
\end{eqnarray*}
Therefore, using \eqref{conzal}, we define $\beta:=\lim_{\delta\to 0}\frac{\mu_\delta}{\Vert \rd\Vert_{\frac{2n}{n+2}}}\in [0,+\infty)$. We define
$$\hat{B}_\delta(X):=\frac{\delta^{\frac{n-2}{2}}B_\delta(\hbox{exp}_{x_0}(\delta X))}{\Vert \rd\Vert_{\frac{2n}{n+2}}}\hbox{ for }X\in B(0, \delta^{-1}i_g(M)).$$
We define the metric $g_\delta:=\hbox{exp}_{x_0}^\star g(\delta\cdot)$ and fix $R>0$. W let $\eta\in C^\infty_c(\rn)$ be such that $\hbox{Supp }\eta\subset B(0,R)$. We have that
\begin{eqnarray*}
|\nabla_{g_\delta}(\eta\hat{B}_\delta)|^2_{g_\delta}(X)&=&\frac{\delta^n}{\Vert \rd\Vert_{\frac{2n}{n+2}}^2}\left|\nabla \left(\eta(\delta^{-1}\hbox{exp}_{x_0}^{-1})B_\delta\right)\right|_g^2(\hbox{exp}_{x_0}(\delta X))\\
&\leq & C\frac{\delta^n}{\Vert \rd\Vert_{\frac{2n}{n+2}}^2}\left(|\nabla B_\delta|^2 +\frac{1}{\delta^2}B_\delta^2\right)(\hbox{exp}_{x_0}(\delta X)){\bf 1}_{|X|<R} 
\end{eqnarray*}
Integrating on $\rn$ and using H\"older and Sobolev inequalities yield
\begin{eqnarray*}
&&\int_{\rn}|\nabla_{g_\delta}(\eta \hat{B}_\delta)|^2_{g_\delta}\, dX  \leq C\frac{\delta^n}{\Vert \rd\Vert_{\frac{2n}{n+2}}^2}\left(\int_{|X|<R}(|\nabla B_\delta|_g^2+\frac{1}{\delta^2} B_\delta^2)(\hbox{exp}_{x_0}(\delta X))\, dX\right)\\
&  &\leq C\frac{1}{\Vert \rd\Vert_{\frac{2n}{n+2}}^2}\left(\int_{B(x_0, R\delta)}(|\nabla B_\delta|_g^2+\frac{1}{\delta^2} B_\delta^2)\, dv_g\right)\\
& & \leq C\frac{1}{\Vert \rd\Vert_{\frac{2n}{n+2}}^2}\left(\int_{B(x_0, R\delta)}|\nabla B_\delta|_g^2\, dv_g+\left(\int_{B(x_0, R\delta)}|B_\delta|^{\frac{2n}{n-2}}\, dv_g\right)^{\frac{n-2}{n}}\right)\\
&& \leq C\frac{\Vert B_\delta\Vert_{H_1^2}^2}{\Vert \rd\Vert_{\frac{2n}{n+2}}^2}\leq C'
\end{eqnarray*}
for all $\delta<<1$. It then follows from reflexivity that $\eta \hat{B}_\delta$ has a weak limit in $\dundeux$ as $\delta\to 0$. A careful analysis yields the existence of $\hat{B}\in H_{1,loc}^2(\rn)$ such that for all $\eta\in C^\infty_c(\rn)$, then $\eta \hat{B}_\delta\rightharpoonup \eta\hat{B}$ weakly in $\dundeux$ as $\delta\to 0$. As one checks, the above computations and some change of variable yield $C>0$ such that 
$$\int_{B(0,R)}|\nabla \hat{B} |^2 \, dX+\int_{B(0,R)}| \hat{B} |^{\frac{2n}{n-2}} \, dX\leq C\hbox{ for all }R>0.$$
Using lower-semi-continuity and pointwise convergence, we get that $|\nabla\hat{B}|\in L^2(\rn)$ and $\hat{B}\in L^{\frac{2n}{n-2}}(\rn)$. As one checks, this yields $\hat{B}\in \dundeux$. Moreover, since $B_\delta\in K_\delta^\perp$, we have that $\l B_\delta,Z_\delta\r_h=0$. With a change of variable and taking $\delta\to 0$, we then get that $\hat{B}\perp Z_0$ in $\dundeux$, and then $\hat{B}\in K_0^\perp$ (see \eqref{spanZ_0}). With a change of variables, equation \eqref{eq:B} rewrites
\begin{eqnarray*}
&&\Delta_{g_\delta} \hat{B}_\delta+\hat{h}_\delta \hat{B}_\delta-(\crit-1)\frac{U_1^{\crit-2}}{|X|^s}\hat{B}_\delta\\
&&=-\frac{\delta^2}{\Vert \rd\Vert_{\frac{2n}{n+2}}}\delta^{\frac{n-2}{2}}\rd(\hbox{exp}_{x_0}(\delta X))+\frac{\mu_\delta}{\Vert \rd\Vert_{\frac{2n}{n+2}}} (\Delta_{g_\delta}Z_0+\hat{h}_\delta Z_0)
\end{eqnarray*}
on all compact subset of $\rn$ for $\delta<<1$, where $\hat{h}_\delta(X):=\delta^2 h(\hbox{exp}_{x_0}(\delta X))$.  Passing to the limit in the equation and using \eqref{exp:Rd} and \eqref{def:alpha} yield
\begin{equation*}
\Delta_{\xi} \hat{B} -(\crit-1)\frac{U^{\crit-2}}{|X|^s}\hat{B} =-\alpha \left(h(x_0)U+\frac{1}{3}R_{ij}(x_0)\sigma^i\sigma^j(r\partial_r U)\right)+\beta \Delta_{\xi}Z_0
\end{equation*}
in the weak sense. Moreover, it follows from standard elliptic theory that $\hat{B}_\delta\to \hat{B}$ in $C^0_{loc}(\rn)$. Since $\hat{B}\in K_0^\perp$, then, multiplying the above equation by $Z_0$ solution to \eqref{eq:Z0} and integrating on $\rn$ yields
$$\alpha \int_{\rn}\left(h(x_0)U+\frac{1}{3}R_{ij}\sigma^i\sigma^j(r\partial_r U)\right)Z_0\, dx=\beta \int_{\rn}|\nabla Z_0|^2\, dx.$$
We then get that for all $R>0$,
\begin{eqnarray*}
\int_{B(x_0,R\delta)}\rd B_\delta\, dv_g&=&\delta^2\int_{B( 0,R )} \delta^{\frac{n-2}{2}}\rd(\hbox{exp}_{x_0}(\delta X)) \delta^{\frac{n-2}{2}} B_\delta(\hbox{exp}_{x_0}(\delta X))\, dv_{g_\delta}
\end{eqnarray*}
and therefore we have that
$$\lim_{\delta\to 0}\frac{\int_{B(x_0,R\delta)}\rd B_\delta\, dv_g}{\delta^2 \Vert \rd\Vert_{\frac{2n}{n+2}}}=\int_{B( 0,R )} \left(h(x_0)U+\frac{1}{3}R_{ij}\sigma^i\sigma^j(r\partial_r U)\right)\hat{B}\, dX.$$
For $n\geq 7$, we get that
$$\lim_{R\to +\infty}\lim_{\delta\to 0}\delta^{-2}\Vert \rd\Vert_{L^{\frac{2n}{n+2}}(M-B(x_0,R\delta))}=0.$$
Noting that 
\begin{eqnarray*}
\int_{M-B(x_0,R\delta)}\rd B_\delta\, dv_g&=&O(\Vert \rd\Vert_{L^{\frac{2n}{n+2}}(M-B(x_0,R\delta))}\Vert B_\delta\Vert_{L^{\frac{2n}{n-2}}(M )})\\
&=&O(\Vert \rd\Vert_{L^{\frac{2n}{n+2}}(M-B(x_0,R\delta))}\Vert \rd\Vert_{L^{\frac{2n}{n+2}}})
\end{eqnarray*}
we get that
\begin{eqnarray*}
\lim_{\delta\to 0}\frac{\int_{M}\rd B_\delta\, dv_g}{\delta^4}=\alpha^{-1}\int_{\rn} \left(h(x_0)U+\frac{1}{3}R_{ij}\sigma^i\sigma^j(r\partial_r U)\right)\hat{B}\, dX.
\end{eqnarray*}
We define $$W_{a,\ricg(x_0)}:=aU+\frac{1}{3}R_{ij}(x_0)\sigma^i\sigma^j(r\partial_r U),$$
so that with \eqref{exp:energy}, we get that
\begin{eqnarray*}
J_h(\ud+\phi_\delta)&=& J_h(\ud)-\frac{\delta^4 }{2}  \int_{\rn} W_{h(x_0),\ricg(x_0)}\hat{C}(W_{h(x_0),\ricg(x_0)})\, dX +o(\delta^{4 })
\end{eqnarray*}
where for any $W\in L^{\frac{2n}{n+2}}(\rn)$, $\hat{C}=\hat{C}(W)\in \dundeux$ is the unique solution to 
\begin{equation*}
\Delta_{\xi} \hat{C} -(\crit-1)\frac{U_1^{\crit-2}}{|X|^s}\hat{C} =-W+\frac{ \int_{\rn}W Z_0\, dx}{ \int_{\rn}|\nabla Z_0|^2\, dx} \Delta_{\xi}Z_0
\,\, ;\, \hat{C}\in K_0^\perp.\end{equation*}
Note that the above expressions make also sense when $h(x_0)=0$ and $\ricg(x_0)=0$.

\medskip\noindent As one checks, we have that $h\mapsto \hat{C}(W_{h(x_0),\ricg(x_0)})$ is continuous from $C^0(M)\to \dundeux$. With Proposition \ref{pro:JWP}, we  get for $n\geq 7$ that 
\begin{eqnarray*}
J_h(\ud+\phi_\delta)&=& \frac{2-s}{2(n-s)}\int_{\rr^n} \frac{U_1^{\crit}}{|X|^s}\, dX+\frac{1}{2}\left(h-c_{n,s}\sg\right)(x_0)\dl^{2}\int_{\rr^n}U_1^2\, dx\\
			&&+L_g(h_0, x_0)\dl^4+o\left( \dl^4\right) +O(|h_0(x_0)-c_{n,s}\sg(x_0)|\delta^4)\\
			&&+O(\Vert h-h_0\Vert_{C^2}\delta^4)
\end{eqnarray*}
where $L_g(h,x_0)$ is defined in \eqref{def:Lg:intro}.  

\medskip\noindent We are in position to prove Theorems \ref{th:2} and \ref{th:1}. Let us assume in addition to our hypotheses that $h_0\in C^p(M)$, $p\geq 2$, and let us assume that
$$h_0(x_0)=c_{n,s}\sg(x_0).$$
We let $f\in C^p(M)$ and we set for $t>0$
$$\delta:=t\sqrt{\eps}\hbox{ and }h:=h_0+\eps f.$$
In particular $\lim_{\eps\to 0}L_g(h_0+\eps f, x_0)=L_g(h_0,x_0)$. Therefore, for all $t>0$, we get
\begin{eqnarray*}
&&\lim_{\eps\to 0}\frac{J_{h_0+\eps f}(U_{t\sqrt{\eps},x_0}+\phi_{t\sqrt{\eps}})- \frac{2-s}{2(n-s)}\int_{\rr^n} \frac{U_1^{\crit}}{|X|^s}\, dX}{\eps^2}\\
&&=\frac{1}{2}  f(x_0)t^2\int_{\rr^n}U_1^2\, dx+L_g(h_0, x_0)t^4,
\end{eqnarray*}
and this convergence is uniform on any compact subset of $(0,+\infty)$.

\smallskip\noindent If $L_g(h_0,x_0)\neq 0$, we assume that $f$ is such that $f(x_0)L_g(h_0,x_0)<0$. Then there exists $t_0\in (0,+\infty)$ that is a nondegenerate local extremum of the RHS. Therefore,  there exists $(t_\eps)_{\eps>0}>0$ such that $\lim_{\eps\to 0}t_\eps=t_0$ and $t_\eps$ is a critical point of $t\mapsto J_{h_0+\eps f}(U_{t\sqrt{\eps},x_0}+\phi_{t\sqrt{\eps}})$, therefore, it follows from Theorem \ref{maintheo1} that $u_\eps:=U_{t_\eps\sqrt{\eps},x_0}+\phi_{t_\eps\sqrt{\eps}}\in H_1^2(M)$ is a weak solution to
$$\Delta_gu_\eps+(h_0+\eps f)u_\eps=\frac{(u_\eps)_+^{\crit-1}}{d_g(x,x_0)^s}\hbox{ in }M.$$
Since $u_\eps-U_{t\sqrt{\eps}}\to 0$ in $H_1^2(M)$, we get that $u_\eps\not\equiv 0$ for $\eps>0$ small enough. Regularity yields (see Jaber \cite{JaberNATM}) that $u_\eps\in C^0(M)$. It then follows from the maximum principle ($\Delta_g+h_0$ is coercive) that $u_\eps>0$ in $M$. Therefore

$$\Delta_gu_\eps+(h_0+\eps f)u_\eps=\frac{u_\eps^{\crit-1}}{d_g(x,x_0)^s}\hbox{ in }M,\, u_\eps>0$$
and $$u_\eps=U_{t_\eps\sqrt{\eps},x_0}+o(1)\hbox{ in }H_1^2(M).$$
This proves Theorem \ref{th:1}. We now prove Theorem \ref{th:2}. If $L_g(h_0,x_0)\neq 0$, it is a consequence of Theorem \ref{th:1}.

\medskip\noindent We now assume that $L_g(h_0,x_0)= 0$. For $k\in\nn$, $k\geq 1$, we define $h_k(x):=h_0(x)+\frac{1}{k}d_g(x,x_0)^2\chi(x)$. As one checks, $h_k(x_0)=h_0(x_0)$ and $\Delta_gh_k(x_0)=\Delta_g h_0(x_0)-\frac{2n}{k}$. Therefore $L_g(h_k, x_0)=L_g(h_0, x_0)+\frac{1}{2k}\int_{\rr^n}|X|^4|\nabla U_1|^2\, dx=\frac{1}{2k}\int_{\rr^n}|X|^4|\nabla U_1|^2\, dx>0$. We impose that $f(x_0)<0$. It then follows from the first case that there there exists a family $u_{k,\eps}\in H_1^2(M)\cap C^0(M)$ such that
$$\Delta_gu_{k,\eps}+\left(h_0+\frac{1}{k}d_g(x,x_0)^2\chi+\eps f\right)u_{k,\eps}=\frac{u_{k,\eps}^{\crit-1}}{d_g(x,x_0)^s}\hbox{ in }M,\, u_{k,\eps}>0$$
and $$u_{k,\eps}=U_{t_{k,\eps}\sqrt{\eps}, x_0}+o(1)\hbox{ in }H_1^2(M)\hbox{ as }\eps\to 0.$$
Therefore, for all $k>1$, there exists $0<\eps_k<\frac{1}{k}$ such that
$$\left|\Vert \nabla u_{k,\eps_k}\Vert_2-\Vert \nabla U_1\Vert_2\right|<\frac{1}{k}\hbox{ and }\Vert u_{k,\eps_k}\Vert_2<\frac{1}{k}.$$
We set $\tilde{u}_k:=u_{k,\eps_k}$ and $\tilde{h}_k:=h_0+\frac{1}{k}d_g(x,x_0)^2\chi+\eps_k f$. We then get that
$$\Delta_g\tilde{u}_k+\tilde{h}_k \tilde{u}_k =\frac{\tilde{u}_k^{\crit-1}}{d_g(x,x_0)^s}\hbox{ in }M,\, \tilde{u}_k>0\hbox{ in }M.$$
As one checks, 
$$\lim_{k\to +\infty} \Vert \tilde{u}_k\Vert_2=0\hbox{ , }\lim_{k\to +\infty} \Vert \nabla\tilde{u}_k\Vert_2=\Vert \nabla U_1\Vert_2\hbox{ and }\lim_{k\to +\infty }\tilde{h}_k=h_0\hbox{ in }C^p(M).$$
So $\tilde{u}_k$ blows-up along a single bubble. This proves Theorem \ref{th:2}. 	

\section{Energy Estimate: the case $n\geq 7$}
We let $h_0\in C^2(M)$ be such that $\dg +h_0$ is coercive. We fix $C_0>0$. Here again, all the estimates are controls are valid for any $h\in C^2(M)$ such that \eqref{hyp:h} holds.
\subsection{Energy Estimates. } The expansion of the Riemannian element of volume is (see Lemma 5.5 page 60 in Lee-Parker \cite{LeeParkerAMS})\begin{eqnarray*}
	\det(g)(x)&=&1-\frac{1}{3}R_{pq} x^px^q-\frac{1}{6}R_{pq,k}x^px^qx^k\\
	&&\hspace{-0.5cm}-\left(\frac{1}{20}R_{pq,kl} 
	+\frac{1}{90}\sum_{i,m}R_{ipqm} R_{iklm}  -\frac{1}{18} R_{pq} R_{kl}  \right)x^px^qx^kx^l+O(r^5),\non
\end{eqnarray*}
where  the \( R_{pq} \) and \( R_{pqkl} \) denote the components of the Ricci tensor \( \ricg \) and the Riemann tensor \( \rmg \) at $x_0$ in the exponential chart. Moreover, we have set $R_{pq,k}=\left( \nabla\ricg \right)_{kpq}(x_0)$ and \(R_{pq,kl}=\left( \nabla^2\ricg (x_0)\right)_{klpq}\). We then get
\begin{eqnarray*}
\sqrt	{\det(g)(x)}&=&1-\frac{1}{6}R_{pq} x^px^q-\frac{1}{12}R_{pq,k} x^px^qx^k\\
	&&\hspace{-1cm}+\frac{1}{24}\left(-\frac{3}{5}R_{pq,kl} 
	-\frac{2}{15}\sum_{i,m}R_{ipqm}R_{iklm}  +\frac{1}{3}\sum_{i,m}R_{pq} R_{kl}  \right)x^px^qx^kx^l+O(r^5).\non
\end{eqnarray*}
We define
\begin{equation}\label{def:G}
	G(r):=\frac{1}{\omega_{n-1}}\int_{S^{n-1}}\sqrt{\det(g)}(r\sigma)\, d\sigma,
\end{equation}
where \( d\sigma \) denotes the volume element on \( \mathbb{S}^{n-1} \), and \( \omega_{n-1} \) denotes the volume of the unit sphere \( \mathbb{S}^{n-1} \).  Then
\begin{equation}\label{Gr}
	G(r)=1-\frac{\sg(x_0)}{6n}r^2+\fg(x_0)r^4+O(r^5),\non
\end{equation}
as $r\to 0$. Here, $\fg$ is defined as follows:
\begin{equation}\label{fgx0}
	\fg (x_0):=\frac{1}{360n(n+2)}\Big(18\dg \sg +8|\ricg |^2_g-3|\rmg |^2_g+5\sg^2 \Big)(x_0).
\end{equation}
\begin{prop}\label{pro:JWP} 
		For $n\geq 7$, we have  that 
			\begin{eqnarray*}
			&&	J_h(\ud)=\frac{2-s}{2(n-s)}\int_{\rr^n} \frac{U_1^{\crit}}{|X|^s}\, dX+\frac{1}{2}(h(x_0)-c_{n,s}\sg(x_0))\dl^{2}\int_{\rr^n}U_1^2\, dx\\
			&&-\frac{1}{4n}\left( \dg (h_0 -\ks\sg)(x_0)-\kg(x_0)\right) \dl^4\int_{\rr^n}|X|^4|\nabla U_1|^2\, dx+o\left( \dl^4\right)\\
			&& +O(|h_0 -c_{n,s}\sg|(x_0)\delta^4)+O(\Vert h-h_0\Vert_{C^2}\delta^4)
		\end{eqnarray*}
			as $\delta\to 0$ uniformly with respect to $h\in C^2(M)$ such that \eqref{hyp:h} holds, where $\kg(x_0)$ and $\ks$ are defined in \eqref{def:k}. 
 \end{prop}
			\noindent{\it Proof of Proposition \ref{pro:JWP}:}  We   estimate each term of $J_h(\ud)$.

\begin{claim}\label{claim:nabla}
We claim as $\delta\to 0$ that for $n\geq 7$,
\begin{eqnarray}\label{estnabla}
	\int_{M}	|\nabla \ud|_g^2\, dv_g&=&\int_{\rr^n} |\nabla U_1|^2\,  dX	-\frac{\sg(x_0)}{6n}\dl^{2}\int_{\rr^n}|X|^2|\nabla U_1|^2\, dX\non\\
&&	+\fg(x_0) \dl^4\int_{\rr^n}|X|^4|\nabla U_1|^2\, dx+o(\dl^4)
\end{eqnarray}
where $U_1$ is defined in \eqref{def:U0}, and $\fg$ is defined in \eqref{fgx0}.
\end{claim}\par 
	\noindent{\it Proof of Claim \ref{claim:nabla}:} Since $U_1$ is radial, we get that $$\left|\nabla U_{\delta, x_0} \right|_g^2(\hbox{exp}_{x_0}(X))=\delta^{-n}|\nabla U_1|^2(\delta^{-1}X).$$
	Moreover, we have that 
\begin{equation*}
\int_{M\backslash B(x_0,r_0)} |\nabla \ud|^2=O\left( \dl^{n-2}\right) \bb{ as } \delta \to 0.
\end{equation*}
We then get that
\begin{eqnarray*}
\int_M|\nabla U_{\dl,x_0}|_g^2\, dv_g&=&\int_{B(x_0,r_0)}|\nabla U_{\dl, x_0}|_g^2\, dv_g+O(\delta^{n-2})\\
&=& \int_{B(0,r_0)}|\nabla U_{\dl,x_0}|_g^2(exp_{x_0}(X))\sqrt{|g|(X)}\, dX+O(\delta^{n-2})\\
%&=& \int_0^{r_0} \int_{\mathbb{S}^{n-1}}|\nabla U_{\dl,x_0}|_g^2(exp_{x_0}(r\sigma ))\sqrt{|g|(r\sigma)} r^{n-1}\,dr\, d\sigma+O(\delta^{n-2})\\
&=&\omega_{n-1} \int_0^{r_0} \delta^{-n}|\nabla U_1|^2(\delta^{-1}r)r^{n-1}G(r)\, dr+O(\delta^{n-2})\\
&=&\omega_{n-1} \int_0^{r_0/\delta}|\nabla U_1|^2(r)r^{n-1}G(\delta r)\, dr+O(\delta^{n-2})
\end{eqnarray*}
where $G$ is as in \eqref{def:G}. With \eqref{Gr}, we then get
\begin{eqnarray*}
&&\int_M|\nabla U_{\dl,x_0}|_g^2\, dv_g=\int_{B(x_0,r_0)}|\nabla U_{\dl, x_0}|_g^2\, dv_g+O(\delta^{n-2})\\
&&=\omega_{n-1} \int_0^{r_0/\delta}|\nabla U_1|^2(r)r^{n-1}\left(1-\frac{\sg(x_0)}{6n}r^2\delta^2+\fg(x_0)r^4\delta^4+O(r^5\delta^5)\right)\, dr\\
&&+O(\delta^{n-2})
\end{eqnarray*}
Since $|\nabla U_1|(r)\leq C r^{1-n}$ for $r>1$, estimating the rest of the various integrals and using that $n\geq 7$, we get Claim \ref{claim:nabla}.\qed

\smallskip \noindent Similarly, as $\delta\to 0$, we get that
 \begin{eqnarray}\label{estudlcrit}
 	\int_M \frac{\ud^{\crit}}{d_g(x,x_0)^s}\, dv_g&=&\int_{\rr^n} \frac{U_1^{\crit}}{|X|^s}\, dX-\frac{\sg(x_0)}{6n}\dl^2\int_{\rr^n} |X|^{2-s}U_1^{\crit}\, dX\non\\
 	&&+\fg(x_0)\dl^4\int_{\rr^n}|X|^{4-s}U_1^{\crit}\, dx+o\left( \dl^4\right) 	\hbox{ for } n\geq 7.
 \end{eqnarray}
\begin{claim}\label{claim:hu2}
  	We claim  that for $n\geq 7$
  	\begin{eqnarray*}
  		\int_M h \ud^2\, dv_g
  		&=&h(x_0)\dl^2\int_{\rr^n}U_{1}^2\, dx\\
  		&& \hspace{-1cm}-\frac{1}{2n}\left( \dg h_0 (x_0)+\frac{1}{3}\sg(x_0)h_0 (x_0)\right)\dl^4	\int_{\rr^n}|X|^2U_1^2\, dX\\
		&&+o\left( \dl^4\right)+O\left( \Vert h-h_0\Vert_{C^2}\dl^4\right)
  	\end{eqnarray*}
  	as $\delta\to 0$ uniformly with respect to $h\in C^2(M)$ such that \eqref{hyp:h} holds.  \end{claim} 
 	\noindent{\it Proof of Claim \ref{claim:hu2}:} We define $\hat{h}(X):=h(\exp_{x_0} (X))$ and $\hat{h}_0(X):=h_0(\exp_{x_0} (X))$ for $X\in B(0, i_g(M))$. As in the previous claim, using the explicit expression of $U_\delta$ and $n\geq 7$, we get that
\begin{eqnarray*}
 	&&\int_M h\ud^2\, dv_g=\int_{B(x_0,r_0)} h\ud^2\, dv_g+O(\delta^{n-2})\\
 &=&\int_{B(0,r_0)} \hat{h}(X)U_{\dl}^2\sqrt{|g|}\, dX +O(\delta^{n-2})\\
&=&\int_{B(0,r_0)} \hat{h}(X)U_{\dl}^2 \, dX + \int_{B(0,r_0)} \hat{h}(x)U_{\dl}^2(\sqrt{|g|}-1)\, dX +O(\delta^{n-2})\\
 &=&\int_{B(0,r_0)} \hat{h}(X)U_{\dl}^2 \, dX + \hat{h}(0)\int_{B(0,r_0)}  U_{\dl}^2(\sqrt{|g|}-1)\, dX \\
 &&+\int_{B(0,r_0)} O( |X|^3) U_{\dl}^2 \, dX+O(\delta^{n-2})\\
 &=&\int_{B(0,r_0)} \hat{h}(X)U_{\dl}^2 \, dX + \hat{h}(0)\int_{B(0,r_0)}  U_{\dl}^2(\sqrt{|g|}-1)\, dX +o(\delta^{4})
 \end{eqnarray*}
 Arguing again as in the previous claim, we get that
 \begin{eqnarray*}
&& \hat{h}(0)\int_{B(0,r_0)}  U_{\dl}^2(\sqrt{|g|}-1)\, dX \\
&&=h(x_0)\omega_{n-1}\int_0^{r_0}  U_{\dl}^2(r)\left(-\frac{1}{6n}\sg(x_0)\delta^2 r^2+O(\delta^4r^4)\right)r^{n-1}\, dr \\
 &&= -\frac{h(x_0)}{6n}\sg(x_0)\delta^4\int_{\rn}|X|^2U_1^2\, dX+o(\delta^4)\\
 &&= -\frac{h_0(x_0)}{6n}\sg(x_0)\delta^4\int_{\rn}|X|^2U_1^2\, dX+o(\delta^4)+O\left( \Vert h-h_0\Vert_{C^2}\dl^4\right)\\
 \end{eqnarray*}
 With the Taylor expansion of $\hat{h}$ at $0$, for $X\in B(0,r_0)$, we write
\begin{eqnarray*}
\hat{h}(X)&=&\hat{h}(0)+X^i\partial_i\hat{h}(0)+\int_0^1X^iX^j(1-t)\partial_{ij}\hat{h}(tX)\, dt\\
&=& \hat{h}(0)+X^i\partial_i\hat{h}(0)+\frac{1}{2}X^iX^j\partial_{ij}\hat{h}_0(0)\\
&&+\int_0^1X^iX^j(1-t)(\partial_{ij}\hat{h}_0(tX)-\partial_{ij}\hat{h}_0(0))\, dt\\
&&+\int_0^1X^iX^j(1-t)(\partial_{ij}\hat{h}  -\partial_{ij}\hat{h}_0 )(tX)\, dt\\
&=& \hat{h}(0)+X^i\partial_i\hat{h}(0)+\frac{1}{2}X^iX^j\partial_{ij}\hat{h}_0(0)\\
&&+O\left(|X|^2\sup_{t\in [0,1]}|\partial_{ij}\hat{h}_0(tX)-\partial_{ij}\hat{h}_0(0)|\right)+O\left(|X|^2 \Vert h-h_0\Vert_{C^2}\right)
\end{eqnarray*}
Using the radial symmetry of $U_\delta$ and the continuity of $\partial_{ij}\hat{h}_0$, we get that 
\begin{eqnarray*}
&&\int_{B(0,r_0)} \hat{h}(X)U_{\dl}^2 \, dX\\
&&=\int_{B(0,r_0)} \left(\hat{h}(0)+x^i\partial_i\hat{h}(0)+\frac{1}{2}x^i x^j\partial_{ij}\hat{h}_0(0) \right)U_{\dl}^2 \, dX +o(\delta^4) \\
&& \qquad  +O(\Vert h-h_0\Vert_{C^2}\delta^4)\\
 &&=h(x_0)\delta^2\int_{B(0,r_0/\delta)} U_{1}^2 \, dX + \frac{1}{2} \partial_{ij}\hat{h}_0(0)\delta^4\int_{B(0,r_0/\delta)} x^i x^jU_{1}^2 \, dX\\
 &&\qquad+o(\delta^4)+O(\Vert h-h_0\Vert_{C^2}\delta^4)\\
 &&=h(x_0)\delta^2\int_{B(0,r_0/\delta)} U_{1}^2 \, dX + \frac{\sum_i\partial_{ii}\hat{h}_0(0)}{2}  \delta^4\int_{B(0,r_0/\delta)} |x|^2U_{1}^2 \, dX\\
 &&\qquad+o(\delta^4)+O(\Vert h-h_0\Vert_{C^2}\delta^4)\\
&&=h(x_0)\delta^2\int_{\rn} U_{1}^2 \, dX-  \frac{\Delta_g h_0(x_0)}{2}  \delta^4\int_{\rn} |x|^2U_{1}^2 \, dX+o(\delta^4)+O(\Vert h-h_0\Vert_{C^2}\delta^4)
\end{eqnarray*}
where we have used again the asymptotic behavior of $U_1$ at infinity and $n\geq 7$. Plugging together these expansions yields Claim \ref{claim:hu2}. \qed \par

\medskip \noindent{\bf End the proof of Proposition \ref{pro:JWP}:} For $n\geq 7$, if follows from \eqref{estnabla}, Claim \ref{claim:hu2} and \eqref{estudlcrit} that 
 	\begin{eqnarray}\label{jhud}
 &&	J_h( \ud) \non\\
 	&=&\frac{1}{2}\int_M\left( |\nabla \ud|_g^2+h\ud^2\right)\, dv_g-\frac{1}{\crit} \int_M \frac{\ud^{\crit}}{d_g(x,x_0)^s} \, dv_g\non\\
&=&\frac{1}{2} \int_{\rr^n} |\nabla U_1|^2\,  dX-\frac{1}{\crit}\int_{\rr^n} \frac{U_1^{\crit}}{|X|^s}\, dX	\\
&&+\frac{1}{2}\left( h(x_0)-\frac{\sg(x_0)}{6n}\frac{\int_{\rr^n}|X|^2|\nabla U_1|^2\, dX}{\int_{\rr^n}U_1^2\, dx} \right.\non\\
&&\left.+\frac{2\sg(x_0)}{\crit6n }\frac{\int_{\rr^n}|X|^{2-s}U_1^{\crit}\, dX}{\int_{\rr^n}U_1^2\, dx}\right) \dl^{2}\int_{\rr^n}U_1^2\, dx\non\\
&&-\frac{1}{4n}\left( \dg h_0(x_0)+\frac{1}{3}\sg(x_0)h_0(x_0)-2n\fg(x_0)\frac{\int_{\rr^n}|X|^4|\nabla U_1|^2\, dX}{\int_{\rr^n}|X|^2U_1^2\, dx}\right. \non\\
&&\left. +\frac{4n}{\crit}\fg(x_0)\frac{\int_{\rr^n}|X|^{4-s}U_1^{\crit}\, dX}{\int_{\rr^n}|X|^2U_1^2\, dx}\right) \dl^4\int_{\rr^n}|X|^2U_1^2\, dx \non\\
&&+o\left( \dl^4\right)+O(\Vert h-h_0\Vert_{C^2}\delta^4) \non
 	\end{eqnarray}
 		as $\delta\to 0$ uniformly with respect to $h\in C^2(M)$ such that \eqref{hyp:h} holds.
 Using again the result of Jaber \cite{JaberNATM}, we obtain that 
 	\begin{equation*}
 		\frac{\int_{\rr^n}|X|^2|\nabla U_1|^2\, dX}{\int_{\rr^n}U_1^2\, dX}=\frac{n(n-2)(n+2-s)}{2(2n-2-s)}
\end{equation*}		
		\begin{equation*}
		 \bb{ and }\frac{\int_{\rr^n}|X|^{2-s}U_1^{\crit}\, dX}{\int_{\rr^n}U_1^2\, dx}=\frac{\k ^{\crit-2}n(n-4)}{2(n-2)(2n-2-s)}.
 	\end{equation*}
This implies that 
 	\begin{equation}\label{fraction0}
 \frac{\int_{\rr^n}|X|^2|\nabla U_1|^2\, dX}{\int_{\rr^n}U_1^2\, dX}-	\frac{2}{\crit }\frac{\int_{\rr^n}|X|^{2-s}U_1^{\crit}\, dX}{\int_{\rr^n}U_1^2\, dx}=6n c_{n,s},
 	\end{equation}
 	where $c_{n,s}$ is defined in \eqref{def:c}. For $p,q\in \rr_+^{*}$ such that $p-q>1$, we set $I_p^q=\int_0^{+\infty}\frac{t^q}{(1+t)^p}\, dt$. From Jaber \cite{JaberNATM}, we have then that
 	\begin{equation*}
 		I_{p+1}^q=\frac{p-q-1}{p}I_p^q \bb{ and } I_{p+1}^{q+1}=\frac{q+1}{p-q-1}I_{p+1}^q=\frac{q+1}{p}I_p^q.
 	\end{equation*}
Using a change of variables, we can compute that  
 \begin{eqnarray*}
 	\int_{\rr^n} |X|^2 U_1^2\, dX&=&w_{n-1}\k^2\int_{0}^{+\infty} \frac{r^{n+1}}{\left(1+r^{2-s} \right)^{\frac{2(n-2)}{2-s}} }\, dr\non\\
 &=&\frac{w_{n-1}}{2-s}\k^2\int_{0}^{+\infty}\frac{t^{\frac{n+s}{2-s}}}{\left( 1+t\right)^{\frac{2(n-2)}{2-s}}  }\, dt=\frac{w_{n-1}}{2-s}\k^2I_{p}^{q},
 \end{eqnarray*}
\begin{equation*}
 	\int_{\rr^n} |X|^4|\nabla U_1|^2\, dX
 	=\frac{w_{n-1}}{2-s}(n-2)^2\k^2I_{p+2}^{q+2}.
 \end{equation*}
 when we have taken $p=\frac{2(n-2)}{2-s}$ and $q=\frac{n+s}{2-s}$. We then get 
 \begin{equation*}
 	\frac{\int_{\rr^n} |X|^4|\nabla U_1|^2\, dX}{\int_{\rr^n} |X|^2 U_1^2\, dX}=(n-2)^2\frac{I_{p+2}^{q+2}}{I_{p}^{q}}
 	=\frac{(n-2)(n+2)(n+4-s)}{2(2n-2-s)}.
 \end{equation*}
Similarly, we   obtain
 \begin{equation*}
 	\frac{\int_{\rr^n}|X|^{4-s}U_1^{\crit}\, dX}{\int_{\rr^n}|X|^2U_1^2\, dx}=\frac{(n-s)(n-6)(n+2)}{2(2n-2-s)},
 \end{equation*}
so that \begin{eqnarray*}
 	\frac{\int_{\rr^n}|X|^4|\nabla U_1|^2\, dX}{\int_{\rr^n}|X|^2U_1^2\, dx}-\frac{2}{\crit}\frac{\int_{\rr^n}|X|^{4-s} U_1^{\crit}\, dX}{\int_{\rr^n}|X|^2U_1^2\, dx}=\frac{(n+2)(n-2)(10-s)}{2(2n-2-s)}.
 \end{eqnarray*}
 Putting these identities together in \eqref{jhud}, we obtain as $\dl\to 0$ that 
 	 	\begin{eqnarray*}
 		&&	J_h(\ud)=\\
 		&&\frac{2-s}{2(n-s)}\int_{\rr^n} \frac{U_1^{\crit}}{|X|^s}\, dX	+\frac{1}{2}\left( h(x_0)-c_{n,s}\sg(x_0)\right) \dl^{2}\int_{\rr^n}U_1^2\, dx\non\\
 		&&-\frac{1}{4n}\left(\dg h_0 (x_0)+\frac{1}{3}\sg(x_0)h_0 (x_0)\right. \\
		&&\left.-\fg(x_0)\frac{n(n+2)(n-2)(10-s)}{2n-2-s}\right)\dl^4\int_{\rr^n}|X|^2U_1^2\, dx\non\\
		&&+o\left( \dl^4\right) +O(\Vert h-h_0\Vert_{C^2}\delta^4) ,\non
 	\end{eqnarray*}
 	where we have used that $U_1$ verifies \eqref{eq:1}. It follows now from \eqref{fgx0} that 
 	\begin{eqnarray*}
 		&&\frac{1}{3}\sg(x_0)(c_{n,s}\sg(x_0))-\fg(x_0)\frac{n(n+2)(n-2)(10-s)}{(2n-2-s)}\non\\
 		&&=-\ks \dg \sg(x_0)\\
		&&-\frac{\ks}{18}\Big(8|\ricg(x_0)|^2_g-3|\rmg(x_0)|^2_g -\frac{5(2-s)}{10-s}\sg(x_0)^2\Big).
 	\end{eqnarray*}Injecting these last value in the estimate of $J_h(\ud)$ above, we get Proposition \ref{pro:JWP} for $n\geq 7$.\par

\end{document}